\documentclass[12pt]{article}

\usepackage{amsmath}
\usepackage{amsfonts}
\usepackage{amssymb}
\usepackage{extarrows}
\usepackage{amscd}
\usepackage{amsthm}
\usepackage{textcomp}
\usepackage{bbm}
\usepackage{color}
\usepackage[linktocpage=true,plainpages=false,pdfpagelabels=false]{hyperref}
\usepackage{euscript}
\usepackage[new]{old-arrows}

\input xypic

\usepackage[matrix,arrow,curve]{xy}

\usepackage{color}

\newcommand{\quash}[1]{}


\newtheorem{defin}{Definition}

\newtheorem{nt}{Remark}
\newtheorem{Th}{Theorem}
\newtheorem{lemma}{Lemma}

\newtheorem{defin-prop}{Definition-proposition}

\newfont{\ssdbl}{msbm8}
\newfont{\sdbl}{msbm9}
\newfont{\dbl}{msbm10 at 12pt}

\newcommand{\oo}{{\cal O}}

\newcommand{\res}{\mathop {\rm res}}

\newcommand{\Spec}{\mathop {\rm Spec}}

\newcommand{\dz}{\mathbb{Z}}

\newcommand{\dr}{\mathbb{R}}

\newcommand{\lrto}{\longrightarrow}

\def\C{{\mathbb C}}

\newcommand{\xlrto}{\xlongrightarrow}

\begin{document}

\author{
Denis V. Osipov
}

\title{Relative analytic reciprocity laws
{\thanks{The  author was supported by the Basic Research Program of the National Research University Higher School of Economics.}}
}
\date{}

\maketitle

\begin{abstract}
We study reciprocity laws involving complex line bundles on fibrations in oriented circles. In particularly, we prove the following reciprocity law.
Let $B$ be a complex manifold and $\pi_i : M_i \to B$ be a fibration in oriented circles, where $i$ runs through a finite set. Let  $L_i$ and $N_i$ be complex line bundles on every $M_i$. The reciprocity law  states that the sum of all $(\pi_i)_* \left(c_1(L_i)  \cup c_1(N_i) \right)$, where $(\pi_i)_*$ is the Gysin map and $c_1$ is the first Chern class, equals zero in $H^3(B, {\mathbb Z})$ when the disjoint union of all $M_i$ is embedded into a holomorphic family of compact Riemann surfaces over the base $B$ such that in every fiber of this family the disjoint union of the embedded circles is the boundary of an embedded compact Riemann surface with boundary, and all $L_i$ and all $N_i$ are restrictions of   holomorphic line bundles on this family.
\end{abstract}

\section{Introduction}

What are the reciprocity laws? Usually the reciprocity laws come from number theory or arithmetic geometry and relate the local and the global class field theory, see e.g. \cite{Se}.

But the reciprocity laws are valid not only for algebraic varieties over finite fields. For example, the simplest well-known reciprocity law is the following.

Let  $X$ be a compact Riemann surface, and
$\omega \in \Omega^1_{{\C}(X)/{\C}} $ be a meromorphic differential $1$-form on $X$. Then the following sum contains only a finite number non-zero terms and
\begin{equation}   \label{diff-form}
\sum_{x \in X}  \, \res\nolimits_x \omega  =0   \, \mbox{.}
\end{equation}

We see that the reciprocity laws connect the local and the global information on a manifold or a variety.

But we can look on equality~\eqref{diff-form} from another point of view. Fix a finite number of points $x_1, \ldots, x_n$ on $X$ and a finite number of complex numbers $a_1, \ldots, a_n$. Suppose that there is a meromorphic
differential $1$-form $\omega$ on $X$ such that $\omega$ is holomorphic on the non-compact Riemann surface  $X \setminus \{x_1, \ldots, x_n  \}$ and for any $1 \le i \le n$ the residue $\res_{x_i} \omega = a_i$. Then
the following equality is satisfied:
\begin{equation}   \label{sum_a}
\sum_{1 \le i \le n} a_i = 0   \, \mbox{.}
\end{equation}

This point of view is also explicit in the more difficult Deligne reciprocity law, where a compact Riemann surface with boundary is considered, and the local invariant is attached  to a  pair of $\C^*$-valued $C^{\infty}$-functions on every connected component of the boundary, see more in Section~\ref{Deligne}  below.

\medskip

Now for any $1 \le i \le n $ let $\pi_i :  M_i \to B$
be a fibration in oriented circles, where $M_i$ and $B$ are  (finite dimensional) $C^{\infty}$-manifolds,   see also Section~\ref{fibr} below. Let $L_i$ and $N_i$ be  complex line bundles on $M_i$. (We do not assume at the moment that $B$ is a complex manifold.) Consider an element
\begin{equation}   \label{elem}
(\pi_i)_* \left( c_1(L_i)  \cup c_1(N_i)   \right)   \,  \in \, H^3(B, \dz)   \, \mbox{,}
\end{equation}
where $c_1$ is the first Chern class with values in $H^2(M_i  , \dz)$ and $$(\pi_i)_* \,  :   \,  H^4(M_i, \dz)   \lrto H^3(B, \dz)$$ is the Gysin map.

Similarly to equality~\eqref{sum_a} we are interested in when the sum
\begin{equation}  \label{sumb}
\sum_{1 \le i \le n}  (\pi_i)_* \left( c_1(L_i)  \cup c_1(N_i)   \right)
\end{equation}
equals zero depending on possible embedding of
$ \bigsqcup_{1 \le i \le n} M_i $
and extensions of all $L_i$ and $N_i$ to a family of Riemann surfaces (with boundary or without boundary) over the base $B$ with some conditions.

In definition~\ref{star}  we formulate some condition $(\bigstar)$   and in Theorem~\ref{th-1} we prove that if complex line bundles on a family over $B$ satisfies this condition then  the sum~\eqref{sumb} equals zero. We call this property the {\em relative analytic reciprocity law}.

Further, in Theorem~\ref{fam} we prove that holomorphic line bundles on  holomorphic families of compact Riemann surfaces (smooth proper holomorphic morphisms of relative dimension $1$ between complex manifolds) give the condition $(\bigstar)$. See the more precise statement in this theorem.

Finally, as corollary of previous theorems,  in Theorem~\ref{Res_laws} we prove that if $B$ is a complex manifold and $\bigsqcup_{1 \le i \le n} M_i $
  is embedded into a  holomorphic family of compact Riemann surfaces over the base $B$ such that in every fiber of this family the disjoint union of the embedded circles is the boundary of an embedded compact Riemann surface with boundary, and all $L_i$ and all $N_i$ are restrictions of   holomorphic line bundles on this family, then the sum~\eqref{sumb} equals zero.

Surprisingly that the pure topological question on the  equality of the sum~\eqref{sumb} to zero depends on the embedding into holomorphic families of compact Riemann surfaces and extension of linear bundles to holomorphic linear bundles on these families.

We also note that we use in the proofs the Deligne reciprocity law.

Besides, in Remark~\ref{rem-last}  we connect an element~\eqref{elem} from the group $H^3(B, \dz)$ (when $n=1$ and $L_1 =N_1$) with the determinant gerbe and the topological Riemann-Roch theorem for circle fibrations from~\cite{Osip25} and \cite{BKTV}.

The previous results also extend to pullbacks
of fibrations in oriented circle, see Remark~\ref{last}.

In conclusion, we note that a pure topological element~\eqref{elem} has an analog in algebraic geometry and is called there the Deligne bracket (or the Deligne pairing), see~\cite{D1} and~\cite[Introduction]{O2}.

 \medskip

 The paper is organized as follows.

 In Section~\ref{Deligne} we recall on the Deligne reciprocity law.

 In Section~\ref{fibr} we give general definitions on fibrations in oriented circles and the Gysin map.

 In Section~\ref{rel} we introduce the condition~$(\bigstar)$ and prove the general relative analytic reciprocity law using this condition.

 In Section~\ref{hol} we consider  holomorphic families of compact Riemann surfaces and prove the recirpocity law using these families.

\section{Deligne reciprocity law}  \label{Deligne}

Recall the Deligne analytic reciprocity law from~\cite{D2}.

Let $\C^* = \C \setminus 0$.
Let $C^{\infty}(S^1, \C^*)$ be the group of $C^{\infty}$-functions from
the circle $S^1 $  to~$\C^*$. We fix an orientation on $S^1$.
There is a bumultiplicative and antisymmetric pairing introduced by A.~A.~Beilinson and P.~Deligne (see~\cite{Be} and~\cite[\S~2.7]{D2})
\begin{gather*}
{\mathbb T} \; \, : \;  \, C^{\infty}(S^1, \C^*) \times C^{\infty}(S^1, \C^*) \lrto \C^*  \, \mbox{,}  \\
{\mathbb T} (f,g) = \exp \left( \frac{1}{2 \pi i} \int_{x_0}^{x_0}  \log f \,  \frac{dg}{g} \right) g(x_0)^{-\nu(f)}  \, \mbox{,}
\end{gather*}
where the integer
\begin{equation*}
\nu(f) = \frac{1}{2 \pi i} \oint_{S^1} \frac{df}{f}
\end{equation*}
is the winding number of $f$,
$x_0$ is any point on $S^1$, $\log f $ is any branch of the logarithm of $f$ on $S^1 \setminus x_0$, and the integral is on $S^1$ from $x_0$ to $x_0$ in the positive direction of orientation.
The complex number ${\mathbb T}(f,g)$
does not depend on the choice of $x_0$ and  the choice of the branch $\log f$ of the logarithm of $f$.

The pairing $\mathbb T$  is related with the $\cup$-product in Deligne cohomology of $C^{\infty}$-manifolds, see, for example, \cite[\S~8.1, \S~10.1]{Osip25}.

\begin{nt}  \em
Since the pairing $\mathbb T$  is bimultiplicative, it factors through the following maps:
$$
 C^{\infty}(S^1, \C^*) \times C^{\infty}(S^1, \C^*) \lrto   C^{\infty}(S^1, \C^*) \otimes_{\dz} C^{\infty}(S^1, \C^*)  \lrto \C^*  \, \mbox{.}
$$
We will denote the induced homomorphism from $C^{\infty}(S^1, \C^*) \otimes_{\dz} C^{\infty}(S^1, \C^*)$ to $\C^*$ also by the letter $\mathbb T$.
\end{nt}

\begin{nt} \em There is the formal, purely algebraic analog of pairing~$\mathbb T$, called the  Contou-Carr\`{e}re symbol, see~\cite[\S~2.9]{D2},   \cite{CC1}, \cite[\S~2]{OZ},
and there is also  the higher-dimensional
 generalizations of the Contou-Carr\`{e}re symbol, see~\cite{OZ,GO}.
\end{nt}

Let $\Sigma$ be a compact Riemann surface with boundary. This means that
$\Sigma$ has an atlas with charts which are open subsets in the closed upper half-plane $\mathbb H$ of $\C$, i.e. $${\mathbb H} = \left\{ z \in \C  \mid {\rm Im}( z) \ge 0 \right\} \, \mbox{.}$$
The transition maps between  the charts have to be holomorphic on the interior and of class $C^{\infty}$ on the boundary. We note that by the Schwarz reflection principle the transition maps will be  analytic on the boundary, and  they give the prolongation of $\Sigma$ to the doubled Riemann surface.

 Consider a decomposition
$$\partial \Sigma = \bigsqcup_{1 \le i \le n}  \gamma_i \, \mbox{,}$$ where every $\gamma_i$ is $S^1$ with the induced orientation from $\Sigma$.

\begin{nt}  \em
For the sake of convenience of presentation in the article we will not assume that a Riemann surface or a compact Riemann surface with boundary is connected, but we will assume that they have a finite number of connected components.
\end{nt}

We recall the following statement from~\cite[Prop.~4.1]{D2}.

\begin{Th}[Analytic reciprocity law]  \label{Del-res}
 Let $F$ and $G$ be two $\C^*$-valued  $C^{\infty}$-functions on a compact Riemann surface $\Sigma$ with boundary  such that they are holomorphic on the interior $\Sigma \setminus \partial \Sigma$ of $\Sigma$.
Let  the $C^{\infty}$-functions $f$ and $g $ be the restrictions of the functions $F$ and $G$ to $\partial \Sigma$ correspondingly. Let ${\mathbb T}_{\gamma_i}$ be the pairing $\mathbb T$ with respect to $\gamma_i = S^1$, where
$1 \le i \le n$. Then there is an equality
$$
\prod_{1 \le i \le n} {\mathbb T}_{\gamma_i} (f , g) =1  \, \mbox{,}
$$
where  for every $1  \le i \le n$ one restricts $f$ and $g$ to $\gamma_i$ to calculate ${\mathbb T}_{\gamma_i} (f,g) $.
\end{Th}

\begin{nt} \em
Since \cite[Prop.~4.1]{D2} states more than Theorem~\ref{Del-res}, we will briefly recall here the argument for the proof of Theorem~\ref{Del-res}.
By $F$ and $G$ we can associate the complex  line bundle $F \cup G$ on $\Sigma$ with connection, see, e.g. \cite[\S~10.1]{Osip25} (this construction corresponds  to the multiplication $\cup$ between the Deligne cohomology groups $H^1_D(\Sigma, {\mathbb Z}(1) )$). The monodromy of this connection along $\gamma_i$ is ${\mathbb T}_{\gamma_i} (f , g)$.
Besides, the curvature of this connection is
$$\frac{-1}{2 \pi i} \, \frac{dF}{F}  \wedge \frac{dG}{G}  \, \mbox{.}$$
Since the functions $F$ and $G$ are holomorphic on the interior  of $\Sigma$, this curvature is zero, i.e. the connection is flat. Therefore $F \cup G$ corresponds to a homomorphism from $\pi_1(\Sigma)$ to $\C^*$ that factors through
$$H_1(\Sigma, \dz) = \pi_1(\Sigma) / [ \pi_1(\Sigma), \pi_1(\Sigma) ] $$
Now it is enough to note that the $1$-cycle $\gamma = \gamma_1 + \ldots + \gamma_n$ equals zero in $H_1(\Sigma, \dz)$,  since $\gamma = \partial \Sigma$.
\end{nt}

\begin{nt} \em
From Theorem~\ref{Del-res} it follows the Weil reciprocity law for a compact Riemann surface or, in other words, a smooth projective algebraic curve over $\C$ (see~\cite[ch.~III, \S~4]{Se}), since ${\mathbb T}_{\sigma}(h_1,h_2)$ coincides with the  tame symbol of meromorphic fucntions $h_1$ and $h_2$ at a point $p$,
where $\sigma$ is a small circle around $p$, see~\cite{Be}, \cite[\S~2]{D2},  and the tame symbol
$$
(h_1, h_2)_p = (-1)^{\nu_p(h_1)\nu_p(h_2)} \left[h_1^{\nu_p(h_2)}/ h_2^{\nu_p(h_1)}  \right](p)  \, \mbox{,}
$$
where $\nu_p(h_i)$ is the order of zero (or the number  opposite to the order of pole) of $h_i$
at~$p$.
\end{nt}

\medskip

We will use the pairing $\mathbb T$ and the reciprocity law from Theorem~\ref{Del-res} in Section~\ref{rel} of this article.

\section{Fibration in oriented circles}  \label{fibr}

Let $B$ be a  finite dimensional $C^{\infty}$-manifold. Let $S^1$ be the circle.

Let $\pi : M \to B$ be a fibration in oriented circles. This means that $\pi$ is a locally trivial fibration (or, equivalently,  by Ehresmann's lemma,   a proper surjective submersion between $C^{\infty}$-manifolds)
with $S^1$ fibers
which has a collection of transitions
 functions with value in the group ${\rm Diff}^+ (S^1)$  of orientation preserving $\C^{\infty}$-diffeomorphisms of~$S^1$.

 \medskip

 For any integer $p \ge 1$ there is the Gysin map
 $$
 \pi_* \; : \: H^p(M, \dz)  \lrto H^{p-1}(M, \dz)  \, \mbox{,}
 $$
 which follows from the Leray spectral sequence
 $$ E^{pq}_2 = H^p (B, R^q \pi_* \dz) \Rightarrow   H^{p+q}(M , \dz)  \, \mbox{,}$$
 and $\pi_*$ is the composition of the following natural maps from this spectral sequence:
$$
\pi_* \; : \; H^p(M, \dz)  \lrto E_{\infty}^{p-1,1}  \subset E_2^{p-1,1} = H^{p-1}(B, R^1 \pi_* \dz ) \lrto H^{p-1}(B, \dz)  \, \mbox{.}
$$
We used here that the sheaf $R^1 \pi_* \dz$ is canonically isomorphic to the constant  sheaf $\dz$ after the fixing of orientation of the fiber $S^1$.

\begin{nt}  \em
After extension of coefficients from the ring $\dz$ to the field $\dr$, the Gysin map
$$
\pi_* \; : \;  H^p(M, \dr)  \lrto H^{p-1}(B, \dr)
$$
is the integration of differential forms along the fiber $S^1$ after the isomorphism of singular cohomology with real coefficients and  the de Rham cohomology.
\end{nt}

\medskip

By $\oo_M$ and $\oo_M^*$ denote the sheaves of  $\C$-valued and  $\C^*$-valued $C^{\infty}$-functions on $M$ correspondingly. By $\oo_B$ and $\oo_B^*$ denote the analogous sheaves on $B$.

Let $L$ be a complex  line bundle on $M$.  The isomorphism class of $L$ corresponds to the element from $H^1(M, \oo_M^*)$.
There is the first Chern class of $L$
$$c_1(L)  \in H^2(M, \dz)  $$
which is the image of the class of  $L$ under
the isomorphism
$$
H^1(M, \oo_M^*)  \lrto H^2(M, \dz)
$$
that follows from the exponential sheaf sequence
\begin{equation}  \label{exp-sheaf}
0 \lrto 2 \pi i \dz  \lrto \oo_M   \xlrto{\exp}   \oo_M^*  \lrto 1  \, \mbox{.}
\end{equation}

Let $L$ and $N$ be  complex  line bundles on $M$. Using the $\cup$-product between the cohomology groups, we have the following expression:
$$
c_1(L)  \cup c_1(N)  \, \in   \,  H^4(M, \dz)  \, \mbox{.}
$$

\section{Relative analytic reciprocity law}  \label{rel}

Let $K$ be an $n$-dimensional  $C^{\infty}$-manifold with boundary.
Recall that a function $\chi$ from  $K$ to ${\mathbb R}^m$ is called a $C^{\infty}$-function if for every point $q \in K$
 there is a chart $(U, \psi)$, where $q \in U  \subset K$  and $\psi$  is a homeomorphism between the open subset $U$  and an open subset of  the closed upper half-space ${\mathbb R}^{n-1}  \times {\mathbb R}_{\ge 0}$  of ${\mathbb R}^n$,   such that the function $\chi \circ \psi^{-1}$
extends to a smooth function on an open subset of ${\mathbb R}^n$.

\begin{defin}  \label{fam-rim}
Let $P$ be a finite dimensional  $C^{\infty}$-manifold with boundary, and  $B$ be a finite dimensional $C^{\infty}$-manifold.
A surjective $C^{\infty}$-map $\tau$ between $P$ and  $B $ is called a $C^{\infty}$-family of compact Riemann surfaces with boundary if the following conditions are satisfied.
\begin{enumerate}
\item  For every point $x \in B$ the fiber $\tau^{-1}(x)$ is a compact Riemann surface with boundary
 and  the intersection $\partial P \cap \tau^{-1}(x)$ is the boundary of $\tau^{-1}(x)$.
\item Consider a decomposition
$$
\partial P  = \bigsqcup_{1 \le i \le n} M_i  \qquad \mbox{and} \qquad
\pi_i = \tau \mid_{M_i}  \, \mbox{,}
$$
where every $M_i$ is a connected manifold. Then $\pi_i : M_i \to B$ is a fibration in oriented circles for every $1 \le i \le n$.
\end{enumerate}
\end{defin}

Note that in this definition we do not demand any relation between the complex structures of the various fibers of $\tau$.

\begin{defin}
Let $\tau :  P   \to B $ be  a $C^{\infty}$-family of compact Riemann surfaces with boundary. A   complex  line $C^{\infty}$-bundle $E$ on $P$ is called a $C^{\infty}$-holomorphic line bundle,
if for every point $x \in B$ the line bundle $E \mid_{\tau^{-1}(x)}$ is holomorphic on the interior of  $\tau^{-1}(x)$.
\end{defin}

\begin{defin}  \label{star}
Let $\tau :  P   \to B $ be  a $C^{\infty}$-family of compact Riemann surfaces with boundary. Let $E$ be a $C^{\infty}$-holomorphic line bundle on $P$. We will say that $E$ satisfies the condition $(\bigstar)$,
if for every point $x  \in B$ there is an open neighbourhood $V$ of $x$ in $B$ and a non-vanishing $C^{\infty}$-section
$$s \;  :  \; \partial(\tau^{-1}(V)) \lrto  E \mid_{\partial(\tau^{-1}(V))}$$
  such that
  for any point $v \in V$ the restriction of  $s$ to ${\partial(\tau^{-1}(v)) }$ can be extended to a non-vanishing section  $$\tau^{-1}(v)  \lrto E \mid_{\tau^{-1}(v)}$$ which is holomorphic on the interior of  $\tau^{-1}(v)$.
\end{defin}

\begin{nt}  \em
The condition $(\bigstar)$ from Definition~\ref{star} is the ``fusion'' of two natural properties. The first property is that
if $\pi : M  \to B$ is a fibration in oriented circles and $L$ is a complex line bundle on $M$, then for every point $x \in B$ there is  an open neighbourhood $V$ of $x$ in $B$ such that $L \mid_{ \pi^{-1}(V)}$ is the trivial line bundle. The second property is that any holomorphic  line bundle on a non-compact connected Riemann surface is trivial, \linebreak see~{\cite[\S~30]{For}}.

Further, in Section~\ref{hol} we give natural examples when the condition $(\bigstar)$ is satisfied.
\end{nt}

\medskip

We give now the relative reciprocity law.
\begin{Th}[Relative analytic  reciprocity law]  \label{th-1}
Let $\tau : P \to B$  be a
$C^{\infty}$-family of compact Riemann surfaces with boundary.
Let $$\partial P = \bigsqcup_{1 \le i \le n} M_i \, \mbox{,}$$
where every $M_i$ is a connected manifold.
For any  $1 \le i \le n$ let
 $\pi_i : M_i \to B$, where  $\pi_i$ is the restriction of $\tau$ to $M_i$, be the corresponding fibration in oriented circles.
Let $E$ and $H$ be $C^{\infty}$-holomorphic line bundles on $P$ that satisfy condition $(\bigstar)$ from Definition~\ref{star}. For any $1 \le i \le n$  let $$L_i = E \mid_{M_i} \qquad  \mbox{and} \qquad  N_i = H \mid_{M_i} \, \mbox{.}$$
Then in the group $H^3(B, \dz)$ we have
\begin{equation}  \label{ident}
\sum_{1 \le i \le n} (\pi_i)_* \, (c_1(L_i) \cup c_1(N_i)) = 0  \, \mbox{.}
\end{equation}
\end{Th}
\begin{proof}
Let $\pi : M  \to B$ be a fibration in oriented circles.

Since $R^q \pi_* \oo_M^* = \{1\}$ for any integer $q \ge 1$, from the  Leray spectral sequence for the  sheaf $\oo_M^*$ we have for any integer $p \ge 0$ the isomorphism
$$
{H^p(M, \oo_M^*) \lrto H^{p}(B, \pi_* \oo_M^* )}  \, \mbox{.}
$$

From the exponential sheaf sequence~\eqref{exp-sheaf} and $\cup$-product in singular cohomology we have the following composition of maps:
\begin{multline}  \label{first-seq}
H^1(M, \oo_M^*)  \times H^1(M, \oo_M^*) \lrto H^{2}(M, \dz)  \times H^{2}(M, \dz) \stackrel{\cup}{\lrto}   H^{4}(M, \dz)  \xlrto{\pi_*}   \\
 \xlrto{\pi_*}   H^{3}(B, \dz)  \, \mbox{.}
\end{multline}

From~\cite[Theorem~9]{Osip25} it follows that the composition of maps~\eqref{first-seq} coincides with the following composition of maps
\begin{multline}   \label{second-sec}
H^1(M, \oo_M^*)  \times H^1(M, \oo_M^*) \lrto  H^1(B, \pi_* \oo_M^*)  \times H^1(B,  \pi_* \oo_M^*)  \stackrel{\cup}{\lrto}   \\
\stackrel{\cup}{\lrto}
 H^{2}(B, \pi_* \oo_M^*  \otimes_{\dz} \pi_* \oo_M^*)   \xlrto{{\mathbb T}}   H^{2}(B, \oo_B^*)  \lrto H^{3}(B, \dz)   \, \mbox{.}
\end{multline}
Here the map ${{\mathbb T}} $ is induced by the corresponding map of sheaves, where on every fiber $\pi^{-1}(x)= S^1$ ($x \in B$) the map ${\mathbb T}$  is applied. Besides, we used in the last map in~\eqref{second-sec} the exponential sheaf sequence on $B$.

Note that for $M = M_i$ and $\pi = \pi_i$, where $1 \le i \le n$, the composition of maps~\eqref{first-seq}   gives the $i$th summand in  the left hand side of equation~\eqref{ident}.

Consider the map of sheaves on $B$:
\begin{equation}  \label{compo}
\left( \prod_{1 \le i \le n }  (\pi_i)_*  \oo_{M_i}^* \right)  \;   \otimes_{\dz}  \; \left( \prod_{1 \le i \le n }  (\pi_i)_*  \oo_{M_i}^* \right) \; \lrto \;
\prod_{1 \le i \le n }  (\pi_i)_*  \oo_{M_i}^*   \otimes_{\dz}  (\pi_i)_*  \oo_{M_i}^*   \xlrto{\prod\limits_{1 \le i \le n } {\mathbb T}} \oo_B^*  \, \mbox{.}
\end{equation}
Here the map $\prod\limits_{1 \le i \le n } {{\mathbb T}} $ is
is the product of maps $\mathbb T$, where for every  $1 \le i \le n$  the map $\mathbb T$ is applied to functions on a fiber $\pi_i^{-1}(x)= S^1$  for each $x \in B$.

Consider a subsheaf of groups on $B$:
$$
\oo_{\rm hol}^*   \;  \subset  \; \prod_{1 \le i \le n }  (\pi_i)_*  \oo_{M_i}^*  \, \mbox{,}
$$
where on every open subset  $U \subset B$ an element
$$(f_1, \ldots, f_n)  \in   \prod_{1 \le i \le n } H^0(U,  (\pi_i)_*  \oo_{M_i}^*) = \prod_{1 \le i \le n }  H^0(\pi_i^{-1}(U), \oo_{\pi_i^{-1}(U)}^* )  $$
belongs to $H^0(U, \oo_{\rm hol}^*)$ if for every $x  \in U $ there is a $\C^*$-valued $C^{\infty}$-function $\psi$ on $\tau^{-1}(x)$
such that $\psi$ is holomorphic on the interior of $\tau^{-1}(x)$ and for any  $1 \le i \le n$ the function $\psi$ coincides with $f_i$ after restriction to $\pi_i^{-1}(x)$.

From Theorem~\ref{Del-res} (the analytic reciprocity law) we immediately obtain that the composition of maps~\eqref{compo} restricted to the subsheaf $\oo_{\rm hol}^*  \otimes_{\dz} \oo_{\rm hol}^*$ is trivial, i.e. it equals $1$.

Besides,
\begin{gather*}
\prod_{1 \le i \le n }  H^1 \left(B, (\pi_i)_* \oo_{M_i}^* \right)  = H^1 \left(B, \prod_{1 \le i \le n } (\pi_i)_* \oo_{M_i}^* \right)   \, \mbox{,}
\\
\prod_{1 \le i \le n } H^2 \left(B, (\pi_i)_* \oo_{M_i}^*   \otimes_{\dz} (\pi_i)_* \oo_{M_i}^* \right)  =
 H^2 \left(B, \prod_{1 \le i \le n } (\pi_i)_* \oo_{M_i}^*  \otimes_{\dz} (\pi_i)_* \oo_{M_i}^* \right)  \, \mbox{.}
\end{gather*}

Now from condition $(\bigstar)$ we have that elements
$$
\prod_{1 \le i \le n } L_i   \qquad \mbox{and} \qquad  \prod_{1 \le i \le n } N_i
$$
from the group $H^1 \left(B, \prod\limits_{1 \le i \le n } (\pi_i)_* \oo_{M_i}^* \right)$ belong to the image of the group
$H^1(B, \oo_{\rm hol}^*)$. (It is easy to see by using the \v{C}ech cohomology, since the first \v{C}ech cohomology group of a sheaf always coincides with the first sheaf cohomology group.)

And the following composition of maps  is trivial:
\begin{multline*}
H^1(B, \oo_{\rm hol}^*)   \times  H^1(B, \oo_{\rm hol}^*)  \xlrto{\cup}  H^2(B, \oo_{\rm hol}^* \otimes_{\dz}  \oo_{\rm hol}^*) \lrto \\ \lrto  H^2 \left(B, \left(\prod\limits_{1 \le i \le n } (\pi_i)_* \oo_{M_i}^* \right)
\otimes_{\dz} \left( \prod\limits_{1 \le i \le n } (\pi_i)_* \oo_{M_i}^* \right) \right)
\lrto  \\ \lrto H^2 \left(B, \prod_{1 \le i \le n } (\pi_i)_* \oo_{M_i}^*  \otimes_{\dz} (\pi_i)_* \oo_{M_i}^* \right)
\xlrto{\prod\limits_{1 \le i \le n } {\mathbb T}}
H^2 \left( B, \oo_B^*  \right)  \, \mbox{.}
\end{multline*}

This finishes the proof.

\end{proof}

\section{Holomorphic families of compact Riemann surfaces}  \label{hol}

In this section we give examples when the condition $(\bigstar)$ from Definition~\ref{star} is satisfied. This condition was important for the relative reciprocity law in Theorem~\ref{th-1}.

Recall that a morphism $\phi$ between complex analytic spaces $X$ and $Y$ is called smooth  if for every point $x \in X $
one can find the following items: open neighbourhoods $U \subset X$ of $x$, $V \subset Y$ of $\phi(x)$ with $\phi(U) =V$, an open subset $Z \subset {\mathbb C}^k$ and an isomorphism
of complex analytic spaces $\psi : U \to Z \times V$ such that the diagram
 $$
\xymatrix{
U  \ar[rr]^{\psi}   \ar[dr]^{\phi}  &&  Z \times V  \ar[dl]_{\rm pr}  \\
&
V
}
$$
commutes, where $\rm pr$ is the corresponding projection.

When $k $ does not depend on $x  \in X$, we will say that the morphism $\phi$ is of {\em relative dimension $k$}.

\begin{nt}  \em
The morphism $\phi$ is smooth if and only if $\phi$ is flat and every fiber of $\phi$ is a complex manifold. Besides, if $Y$ is a complex manifold and $\phi : X \to Y$ is smooth, then $X$ is a complex manifold.
In addition, for a morphism $\varphi : X' \to Y' $ between connected complex manifolds the following conditions are equivalent:
\begin{enumerate}
\item $\varphi$ is flat;
\item $\varphi$ is open;
\item Every fiber of $\varphi$ is of pure dimension $\dim X' - \dim Y'$.
\end{enumerate}
\end{nt}
See, e.g., \cite[\S~2.18, \S~3.20-3.21]{Fisch} (but note that  in this book a smooth morphism is called a submersion).

\bigskip

Let $\tau :  P   \to B $ be  a $C^{\infty}$-family of compact Riemann surfaces with boundary, and  let $E$ be a $C^{\infty}$-holomorphic line bundle on $P$  (see section~\ref{rel}).
Let $Q$ be a holomorphic line bundle on a complex manifold $W$.

By a {\em morphism} $\rho : P \to W$ such that $\rho^* Q \simeq E$ we mean a $C^{\infty}$-map $\rho$ which is holomorphic on the interior of $\tau^{-1}(x)$ for every point $x \in B$ and such that there is an isomorphism $\rho^* Q \simeq E$ as line $C^{\infty}$-bundles which is an isomorphism of holomorphic line bundles after the restriction to the interior of $\tau^{-1}(x)$ for every point $x \in B$.

\medskip

The following theorem gives the sufficient condition
when the condition $(\bigstar)$ is satisfied. This is important for application  in Theorem~\ref{th-1}.

\begin{Th} \label{fam}
Let $\tau :  P   \to B $ be  a $C^{\infty}$-family of compact Riemann surfaces with boundary, and  let $E$ be a $C^{\infty}$-holomorphic line bundle on $P$.
Suppose that for any point $x \in B$ we have
\begin{itemize}
\item[1)]  an open  neighbourhood $\widetilde{V}$ of $x$ in $B$ and a $C^{\infty}$-map $\varpi$ from $\widetilde{V}$ to a complex connected manifold  $V$,
\item[2)] a  smooth proper morphism $\phi : U \to V$ of relative dimension $1$ between  complex manifolds,
\item[3)] a holomorphic line bundle $Q$ on $U$,
\item[4)] a morphism
 $\rho \, :  \, \tau^{-1}(\widetilde{V}) \to U$ such that $\rho^* Q \simeq E \mid_{\tau^{-1}(\widetilde{V})}$ and the diagram
 $$
\xymatrix{
\tau^{-1}(\widetilde{V})  \ar[rr]^{\rho}   \ar[d]^{\tau \mid_{\tau^{-1}(\widetilde{V})}}  && U \ar[d]_{\phi}  \\
\widetilde{V} \ar[rr]^{\varpi}  && V
}
$$
 commutes.
 \end{itemize}
Then $E$ satisfies the condition $(\bigstar)$.
\end{Th}
\begin{nt}  \em
The map  $\phi$ is surjective and every fiber of $\phi$ is a compact Riemann surface, i.e. a smooth projective algebraic curve over $\mathbb C$.
\end{nt}
\begin{nt} \em
During the proof of this theorem and Lemma~\ref{Stein} below, for any analytic space $Y$ we will denote by $\oo_Y$ the structure sheaf of $Y$. (Note that in Sections~\ref{fibr} and~\ref{rel}  we denoted by the same letter $\oo$ the sheaf of  $\C$-valued $C^{\infty}$-functions on a $C^{\infty}$-manifold, but now we will not use this notation.)
\end{nt}
\begin{proof}
We note that by Ehresmann's lemma, the morphism $\phi$ is a locally trivial $C^{\infty}$- fibration.
Besides, without loss of generality we can assume that $U$ is connected.

By shrinking $V$ and $\widetilde{V}$ to  smaller open subsets  (which we denote by the same letter)  if necessary,
we will assume that  there is an isomorphism  $U \simeq V \times \phi^{-1}(x)$ as  $C^{\infty}$-manifolds.

Again  by shrinking  $V$  and  $\widetilde{V}$ to  smaller open subsets, and  since $\phi$ is a smooth morphism, we will assume that there is a section $s : V \to U $ such that the divisor  $H = s(V)$ does not intersect the set $\rho ( \tau^{-1}(\widetilde{V}))$. The divisor $H$ is a Cartier divisor on~$U$.

Every fiber of $\phi$ is a connected smooth projective algebraic curve over $\mathbb C$. Therefore for every point $y$ on the fiber $\phi^{-1}(z)$, where a point $z \in V$,  the invertible sheaf $\oo_{\phi^{-1}(z)}(y)$
is an ample sheaf on  $\phi^{-1}(z)$. Therefore, the invertible sheaf $\oo_U(H)$ is an $\phi$-ample sheaf on $U$, see \cite[Chapter II, \S~1.10]{N2}, \cite[Prop.~1.4]{N1}. This is equivalent to the statement that  there exist an open covering $V = \bigcup V_{\lambda}$   and closed (analytic) immersions
$$
\varphi_{\lambda}  \, : \, \phi^{-1}(V_{\lambda}) \, \longhookrightarrow \, {\mathbb P}^{n_{\lambda}} \times V_{\lambda}
$$
over $V_{\lambda}$ for some $n_{\lambda}  \in {\mathbb N}$ such that
$$
\oo_{\phi^{-1}(V_{\lambda})}(m_{\lambda} H) \,  \simeq \, \varphi_{\lambda}^* \, p_1^* \oo_{{\mathbb P}^{n_{\lambda}}}(1)
$$
 for some $m_{\lambda}  \in {\mathbb N}$, where $p_1$ is the projection to ${\mathbb P}^{n_{\lambda}}$.

Changing  $V$ and $\widetilde{V}$ again, we will assume that $V$ is $V_{\lambda}$ and $U$ is $\phi^{-1}(V_{\lambda})$ for some~$\lambda$. We denote $n = n_{\lambda}$, $m = m_{\lambda}$ and $\varphi = \varphi_{\lambda}$.

Let $J$ be the coherent $\oo_{{\mathbb P}^{n} \times V}$-ideal  sheaf that defines the closed analytic subspace $\varphi ( U )$ in  ${\mathbb P}^{n} \times V$.
 We have an exact sequence of sheaves on ${\mathbb P}^{n} \times V$
 $$
 0 \lrto J \lrto \oo_{{\mathbb P}^{n} \times V}  \lrto \varphi_* \oo_{U}   \lrto 0 \, \mbox{.}
 $$

From this exact sequence we have the following exact sequence of sheaves on ${\mathbb P}^{n} \times V$ for any integer $l$:
$$
0 \lrto J \otimes_{\oo_{{\mathbb P}^{n} \times V}} p_1^* \oo_{{\mathbb P}^{n}}(l)  \lrto p_1^* \oo_{{\mathbb P}^{n}}(l)  \lrto  \varphi_* \oo_{U}(lm H)  \lrto 0   \, \mbox{.}
$$
This leads to an exact sequence of sheaves on $V$
\begin{multline*}
0 \lrto (p_1)_* \left( J \otimes_{\oo_{{\mathbb P}^{n} \times V}} p_1^* \oo_{{\mathbb P}^{n}}(l)  \right)   \lrto
(p_1)_* p_1^* \oo_{{\mathbb P}^{n}}(l)  \lrto (p_1)_*  \varphi_* \oo_{U}(lm H) \lrto \\ \lrto R^1 (p_1)_* \left( J \otimes_{\oo_{{\mathbb P}^{n} \times V}} p_1^* \oo_{{\mathbb P}^{n}}(l)  \right)  \, \mbox{.}
\end{multline*}

Again  by shrinking  $V$ and $\widetilde{V}$ to smaller open subsets,  by the Grauert and Remmert theorem (see, e.g.,    \cite[Chapter~IV, \S~2]{BS}), there is an integer $k > 0$ such that
$$
R^1 (p_1)_* \left( J \otimes_{\oo_{{\mathbb P}^{n} \times V}} p_1^* \oo_{{\mathbb P}^{n}}(k)  \right) = 0  \, \mbox{.}
$$
Therefore we have the following exact sequence of sheaves on $V$:
\begin{equation}  \label{ex-sh}
0 \lrto (p_1)_* \left( J \otimes_{\oo_{{\mathbb P}^{n} \times V}} p_1^* \oo_{{\mathbb P}^{n}}(k)  \right)   \lrto
(p_1)_* p_1^* \oo_{{\mathbb P}^{n}}(k)  \lrto (p_1)_*  \varphi_* \oo_{U}(km H) \lrto  0  \, \mbox{.}
\end{equation}

Again  by shrinking  $V$  and $\widetilde{V}$ to  smaller open subsets, we can assume that $V$ is a Stein contractible manifold. Since $p_1$ is a proper morphism,  by the Grauert direct image theorem (see, e.g., \cite[Chapter~III, \S~2]{BS})  the sheaf
$ (p_1)_* \left( J \otimes_{\oo_{{\mathbb P}^{n} \times V}} p_1^* \oo_{{\mathbb P}^{n}}(k)  \right)  $ is coherent. Therefore, by Cartan theorem we have
$$
H^1 \left(V, (p_1)_* \left( J \otimes_{\oo_{{\mathbb P}^{n} \times V}} p_1^* \oo_{{\mathbb P}^{n}}(k)  \right) \right)  = 0   \, \mbox{.}
$$
Hence and from~\eqref{ex-sh} we have that the natural map
$$
H^0 \left({\mathbb P}^{n} \times V, \,  p_1^* \oo_{{\mathbb P}^{n}}(k) \right)  \xlrto{\theta}   H^0 \left( U ,\oo_{U}(km H) \right)
$$
is surjective.

Let $1 \in H^0 \left( U ,\oo_{U}(km H) \right)$ be the natural element
that is the image of the element $1 \in H^0 \left( U ,\oo_{U} \right)$.
The zero locus of the section $1 \in H^0 \left( U ,\oo_{U}(km H) \right)$  has the support $H$ as a complex manifold.
Let $h \in H^0 \left({\mathbb P}^{n} \times V, \,  p_1^* \oo_{{\mathbb P}^{n}}(k) \right) $ such that $\theta(h)=1$.
Let the open subset $D_h  \subset {\mathbb P}^{n} \times V$ be the  complement to the zero locus of the section $h$ on ${\mathbb P}^{n} \times V$. Then the morphism $\varphi$ restricted to $U \setminus H$ is a closed immersion
$$
U \setminus H    \longhookrightarrow D_h  \, \mbox{.}
$$

We claim that $D_h$ is a Stein manifold. This is proved (independently of the proof of this theorem) in Lemma~\ref{Stein} after this theorem.
Since any closed complex analytic subspace of a Sten space is again a Stein space (see, e.g., \cite[Chapter V, \S~1]{GR}), the complex manifold $U \setminus H $ is Stein.

By the Oka principle (see~\cite[Chapter~5, \S~2.5]{GR}),
$$
H^1(U \setminus H, \, \oo^*_{U \setminus H})  = H^2(U \setminus H,  \, \dz)  \, \mbox{.}
$$
(This also follows from the holomorphic exponential sheaf sequence on $U \setminus H$.)

Now, using the Mayer-Vietoris sequence for $U$ and the Mayer-Vietoris sequence for $\phi^{-1}(x)$, we have
$$
H^2\left(U \setminus H, \, \dz\right)  = H^2\left( \phi^{-1}(x)  \setminus (H \cap \phi^{-1}(x) ), \, \dz\right) = 0  \, \mbox{.}
$$

Therefore the holomorphic line bundle $Q$ (from the condition of this theorem) restricted to $U \setminus H$ is trivial.
A trivialization of this bundle gives the trivialization of  the line bundle $E \mid_{\tau^{-1}(\widetilde{V})}$
which is necessary for the condition~$(\bigstar)$ to be fulfilled.
Hence the line bundle $E$ satisfies the condition $(\bigstar)$.
\end{proof}

\medskip

\begin{lemma}  \label{Stein}
Let $V$ be a complex analytic Stein space. Consider any section
$$h \in H^0 \left({\mathbb P}^{n} \times V, \,  p_1^* \oo_{{\mathbb P}^{n}}(k) \right) \, \mbox{,}$$
where $k > 0$,  the sheaf  $\oo_{{\mathbb P}^{n}}(k)  = \oo_{{\mathbb P}^{n}}(1)^{\otimes k} $ is an invertible (holomorphic) sheaf on   ${\mathbb P}^{n}$, and $p_1$ is the projection to ${\mathbb P}^{n}$.
Let the open subset
$$D_h  \subset {\mathbb P}^{n} \times V$$
 be the  the complement to the zero locus of the section $h$ on ${\mathbb P}^{n} \times V$. Then $D_h$ is a Stein space.
\end{lemma}
\begin{proof}
Using the \v{C}ech cohomology or the simplest case of the K\"unneth formula for coherent analytic sheaves (see~\cite{KV}, \cite[Chapter~IX, \S~5.B]{Dem}), we have
$$
H^0 \left({\mathbb P}^{n} \times V, \,  p_1^* \oo_{{\mathbb P}^{n}}(k) \right)  =    H^0 \left({\mathbb P}^{n}, \oo_{{\mathbb P}^{n}}(k) \right)  \otimes_{\mathbb C} H^0  \left( V, \oo_V    \right)  \, \mbox{.}
$$

Let the section $h$ has a decomposition
\begin{equation}  \label{rhs}
h = \sum_{ 1 \le i  \le e} b_i \otimes a_i   \, \mbox{,}
\end{equation}
where $b_i \in  H^0 \left({\mathbb P}^{n}, \oo_{{\mathbb P}^{n}}(k) \right) $ and $ a_i \in H^0  \left( V, \oo_V    \right)$. Consider a polynomial ring
$$
A = \C[a_1, \ldots, a_e]   \, \mbox{.}
$$
We have the natural homomorphism of $\C$-algebras $A  \to H^0  \left( V, \oo_V    \right)$. This homomorphism induces the morphism
$$
\alpha  \, : \, V  \lrto (\Spec A)^{\rm an}   \, \mbox{,}
$$
where for every separated  scheme $T$ of finite type over $\C$ by $T^{\rm an }$ we denote the associated analytic space.

Let $\tilde{h}$  be an element from  $H^0 \left({\mathbb P}^{n} \times \Spec A, \,  p_1^* \oo_{{\mathbb P}^{n}}(k) \right)$ given by the right hand side of formula~\eqref{rhs}. Let the open subscheme
$$
D_+ (\tilde{h}) \, \subset  \,  {\mathbb P}^{n} \times \Spec A
$$
be the  the complement to the zero locus of the section $\tilde{h}$ on the scheme ${\mathbb P}^{n} \times \Spec A$.
It is well-known that $D_+(\tilde{h})$ is an affine scheme.
Hence $(D_+(\tilde{h}))^{\rm an}$ is a Stein space.

We have
$$
D_h \, \simeq \, (D_+(\tilde{h}))^{\rm an}  \times_{(\Spec A)^{\rm an}} V   \, \mbox{,}
$$
where we used the morphism $\alpha$.

Now $(D_+(\tilde{h}))^{\rm an}  \times V$ is a Stein space. The space
${(D_+(\tilde{h}))^{\rm an}  \times_{(\Spec A)^{\rm an}} V}$ is a closed analytic  subspace in  $(D_+(\tilde{h}))^{\rm an}  \times V$, since it is the preimage (in the category of analytic spaces) of the diagonal in ${(\Spec A)^{\rm an}  \times (\Spec A)^{\rm an}}$ under the natural morphism. Hence, $D_h$ is a Stein space. (See in~\cite[Chapter V, \S~1]{GR}   the corresponding facts on Stein spaces, which we used.)
\end{proof}

\bigskip

As an application of Theorems~\ref{th-1} and~\ref{fam} we have the following reciprocity law.
\begin{Th}[Reciprocity law inside a family of compact Riemann surfaces] \label{Res_laws}
Let $B$ be a complex manifold.  Let ${\pi_i : M_i \to B}$ be a fibration in oriented circles, where $1 \le i \le n$.
\quash{Let
$$
M = \bigsqcup_{1 \le i \le n} M_i    \qquad \mbox{and} \qquad  \pi = \bigsqcup_{1 \le i \le n}  \pi_i  \, :  \, M \lrto B  \, \mbox{.}
$$}
Let $\phi  : X \to B$  be a surjective smooth proper morphism of  relative dimension $1$ between complex manifolds, and $\bigsqcup_{1 \le i \le n} M_i \longhookrightarrow X$
be a $C^{\infty}$-embedding such that the following diagram
$$
\xymatrix{
 \bigsqcup_{1 \le i \le n} M_i \, \ar@{^{(}->}[rr]   \ar[d]^{\bigsqcup_{1 \le i \le n}  \pi_i}  &&  X \ar[d]_{\phi}  \\
B  \ar@{=}[rr]  && B
}
$$
commutes and for any $x \in B$ the union of circles $ \bigsqcup_{1 \le i \le n} \pi_i^{-1}(x)$ is the boundary of a compact Riemann surface with boundary, where this surface is  embedded into the compact Riemann surface $\phi^{-1}(x)$.
Let $Q$ and $S$ be holomorphic line bundles on $X$, and let
$$
L_i = Q \mid_{M_i}     \mbox{,} \qquad N_i = S \mid_{M_i}
$$
be line bundles on $M_i$, where $1 \le i \le n$.
Then in the group $H^3(B, \dz)$ we have
$$
\sum_{1 \le i \le n} (\pi_i)_* \, (c_1(L_i) \cup c_1(N_i)) = 0  \, \mbox{.}
$$
\end{Th}
\begin{proof}
We have a family of compact Riemann surfaces with boundary inside $X$.
The boundary of this family is $\bigsqcup_{1 \le i \le n} M_i$. We have the restrictions of line bundles $Q$ and $S$ to this family.
By Theorem~\ref{fam}  these restrictions satisfy  the condition $(\bigstar)$. Then we apply Theorem~\ref{th-1}.
\end{proof}

\begin{nt}   \label{rem-last}  \em
Let $\pi : M \to B$ be a fibration in oriented circles, and $L$ be a complex line bundle on $M$. In~\cite{Osip25}  we constructed the determinant gerbe ${\mathcal Det}(L)$
with the class $[{\mathcal Det}(L)]$ in $H^3(B, \dz)$ and proved the following topological Riemann-Roch theorem with values in the group $H^3(B, \dz)$:
$$
12 \, [{\mathcal Det}(L)] = 6  \, \pi_* (c_1(L) \cup c_1(L))  \, \mbox{.}
$$
(Note that in the group $H^3(B, \C)$ the equality $2  [{\mathcal Det}(L)] =    \pi_* (c_1(L) \cup c_1(L))$ was proved earlier by another methods in~\cite{BKTV}.)

Now if $B$ is a complex manifold, and $\pi : M \to B$ is embedded into a  holomorphic family of compact Riemann surfaces over $B$ (i.e. a surjective smooth proper  morphism of relative dimension $1$ to $B$) such that in every fiber of this family a fiber of $\pi$ is the boundary of an embedded compact Riemann surface with boundary, and $L$ can be extended to a holomorphic line bundle on this family, then from Theorem~\ref{Res_laws}  it follows that
$$
12 \, [{\mathcal Det}(L)]  =0  \, \mbox{.}
$$

The investigation of the determinant gerbe ${\mathcal Det}(L)$ in this case was the starting point of this paper.
\end{nt}

\begin{nt} \label{last} \em
The previous results immediately extend to pullbacks of complex line bundles
to fiber products $\bigsqcup_{1 \le i \le n} M_i \times_{\widetilde{B}} B$ and $M \times_{\widetilde{B}} B$, where $\bigsqcup_{1 \le i \le n} M_i \lrto B$ and $M \to B$ were used in Theorem~\ref{Res_laws} and Remark~\ref{rem-last}, and $\widetilde{B}$ is a finite-dimensional \linebreak $C^{\infty}$-ma\-ni\-fold together with a $C^{\infty}$-map $\widetilde{B} \to B$. This follows from the fact that the Gysin map, the Chern classes, and their $\cup$-products commute with these pullbacks.
\end{nt}

\vspace{0.3cm}

\noindent Steklov Mathematical Institute of Russsian Academy of Sciences, 8 Gubkina St., Moscow 119333, Russia,
 {\em and}

 \noindent National Research University Higher School of Economics, Laboratory of Mirror Symmetry,  6 Usacheva str., Moscow 119048, Russia,
{\em and}

\noindent National University of Science and Technology ``MISiS'',  Leninsky Prospekt 4, Moscow  119049, Russia

\noindent {\it E-mail:}  ${d}_{-} osipov@mi{-}ras.ru$


\begin{thebibliography}{99}



\bibitem{Be} A.~A.~Beilinson, {\em
Higher regulators and values of $L$-functions of curves.}
Funktsional. Anal. i Prilozhen.   14 (1980), no. 2, 46--47; english transl. in
Functional Analysis and Its Applications, 14 (1980),  no. 2,  116--118.


\bibitem{BS}
C. B\u{a}nic\u{a},  O. St\u{a}n\u{a}\c{s}il\u{a}, {\em
Algebraic methods in the global theory of complex spaces.}
Translated from the Romanian
Editura Academiei, Bucharest; John Wiley \& Sons, London-New York-Sydney, 1976.

\bibitem{BKTV}
P.~Bressler,  M.~Kapranov, B.~Tsygan, E.~Vasserot,
 {\em
 Riemann-Roch for real varieties},
Progr. Math., 269,
Birkh\"auser Boston, Ltd., Boston, MA, 2009, 125--164.





\bibitem{CC1}
C. Contou-Carr\`{e}re,  {\em Jacobienne locale, groupe de bivecteurs de Witt universel, et symbole mod\'{e}r\'{e},}
C. R. Acad. Sci. Paris S\'er. I Math., {\bf 318}:8 (1994), 743--746.


\bibitem{D1}
P.~Deligne, {\em
Le d\'{e}terminant de la cohomologie.} (French)  Current trends in arithmetical algebraic geometry (Arcata, Calif., 1985), 93--177,
Contemp. Math., 67, Amer. Math. Soc., Providence, RI, 1987.

\bibitem{D2}
P.~Deligne, {\em Le symbole mod\'{e}r\'{e},}   Inst. Hautes \'{E}tudes Sci. Publ. Math. No. 73 (1991), 147--181.


\bibitem{Dem} J.-P.~Demailly, {\em Complex analytic and differential Geometry},  OpenContent book, 2012, Universit\'{e} de Grenoble I, available at
\href{https://www-fourier.univ-grenoble-alpes.fr//~demailly/manuscripts/agbook.pdf}{https://www-fourier.univ-grenoble-alpes.fr//$\sim$demailly/manuscripts/agbook.pdf}




\bibitem{Fisch}
G.~Fischer, {\em Complex analytic geometry}, Lecture Notes in Math., Vol. 538,
Springer-Verlag, Berlin-New York, 1976.




\bibitem{For}
O.~Forster, {\em Lectures on Riemann surfaces}, Grad. Texts in Math., 81
Springer--Verlag, New York-Berlin, 1981.


\bibitem{GO} S.~O.~Gorchinskiy, D.~V.~Osipov, {\em A higher-dimensional Contou-Carr\`{e}re symbol: local
 theory.} (Russian) Mat. Sb. 206 (2015), no. 9, 21--98; translation in Sb. Math. 206
 (2015), no. 9--10, 1191--1259.


\bibitem{GR}  H.~Grauert,  R.~Remmert, {\em
Theory of Stein spaces.}
Translated from the German by Alan Huckleberry,
Grundlehren der Mathematischen Wissenschaften, 236
Springer-Verlag, Berlin-New York, 1979.


\bibitem{KV}
R.~Kiehl, J.-L. Verdier, {\em
Ein einfacher Beweis des Kohärenzsatzes von Grauert.} (German)
Math. Ann. 195 (1971), 24--50.


\bibitem{N1}  N.~Nakayama,  {\em The lower semicontinuity of the plurigenera of complex varieties},
Adv. Stud. Pure Math., 10,
North-Holland Publishing Co., Amsterdam, 1987, 551--590.


\bibitem{N2} N.~Nakayama, {\em Zariski-decomposition and abundance}, MSJ Mem., 14
Mathematical Society of Japan, Tokyo, 2004.



\bibitem{O2} D.~V.~Osipov, {\em Local analog of the Deligne--Riemann--Roch isomorphism for line bundles in relative dimension $1$}, Izvestiya: Mathematics, 2024, Volume 88, Issue 5,  930--976.


\bibitem{Osip25}
D. V. Osipov, {\em Analytic diffeomorphisms of the circle and topological Riemann–Roch theorem for circle fibrations},
Tr. Mat. Inst. Steklova 330 (2025), 227--272; English translation in  Proc. Steklov Inst. Math. 330 (2025); e-print arXiv: 2503.10517.





\bibitem{OZ}
D.~Osipov, X.~Zhu, {\em  The two-dimensional Contou-Carr\`{e}re symbol and reciprocity laws,} J. Algebraic Geom. 25 (2016), no. 4, 703--774.




\bibitem{Se} J.-P. Serre,  {\it Algebraic groups and class fields}, Transl. of the French edition.
 (English)  Graduate Texts in Mathematics, 117. New York etc.:
Springer-Verlag, (1988).









\end{thebibliography}
\end{document}